\documentclass[a4paper, 12pt, draft]{amsart}
\allowdisplaybreaks[4]

\usepackage{amsmath,amssymb,amsthm}
\usepackage{amscd}
\usepackage{amsmath}
\usepackage{amssymb}
\usepackage{amsfonts}
\usepackage{cases}
\usepackage{graphicx}
\usepackage{color}
\usepackage{mathrsfs}
\usepackage{amsthm}
\usepackage{enumerate}
\usepackage[all]{xy} 
\usepackage{setspace}
\usepackage{latexsym}
\usepackage{wrapfig}
\usepackage{url}
\usepackage{array,enumerate}

\DeclareMathAlphabet{\mathpzc}{OT1}{pzc}{m}{it}
\theoremstyle{definition}
\newtheorem{defi}{Definition}[section]
\newtheorem{exam}[defi]{Example}

\theoremstyle{theorem}
\newtheorem{theo}[defi]{Theorem}
\newtheorem{prop}[defi]{Proposition}
\newtheorem{coro}[defi]{Corollary}
\newtheorem{lemm}[defi]{Lemma}

\newtheorem{conj}[defi]{Conjecture}

\theoremstyle{remark}
\newtheorem{rema}[defi]{Remark}

\makeatletter

\@addtoreset{equation}{section}
\title{Hypergroups and distance distributions of random walks on graphs}
\author[K. Endo]{Kenta Endo}
\author[I. Mimura]{Ippei Mimura}
\author[Y. Sawada]{Yusuke Sawada}
\address{K. Endo: Graduate school of mathematics, Nagoya University, Chikusaku, Nagoya, 464-8602, Japan}
\email{m16010c@math.nagoya-u.ac.jp}
\address{I. Mimura: Graduate school of mathematics, Nagoya University, Chikusaku, Nagoya, 464-8602, Japan}
\email{m16052v@math.nagoya-u.ac.jp}
\address{Y. Sawada: Graduate school of mathematics, Nagoya University, Chikusaku, Nagoya, 464-8602, Japan}
\email{m14017c@math.nagoya-u.ac.jp}
\begin{document}
\begin{abstract}
Wildberger's construction enables us to obtain a hypergroup from a special graph via random walks. We will give a probability theoretic interpretation to products on the hypergroup. The hypergroup can be identified with a commutative algebra whose basis is transition matrices. We will estimate the operator norm of such a transition matrix and clarify a relationship between their matrix products and random walks.
\end{abstract}
\maketitle
\section{Introduction}
The concept of hypergroup is the probability theoretic extension of the notion of locally compact group introduced by Dunkl \cite{dunk73}, Jewett \cite{jewe75} and Spector \cite{spec78}. We refer the reader to the monograph \cite{bloo-heye95} by Bloom-Heyer for details of the general theory of (locally compact) hypergroups. 
Discrete hypergroups are generalizations of discrete groups. 
As was the case with groups, 
we can completely determine structures of hypergroups of low order. 
Structures of finite hypergroups of low orders have been studied in \cite{wild02} and \cite{stmoy17} for examples. 
In this paper, we shall treat hermitian discrete hypergroups which are generalizations of $\mathbb{Z}/2\mathbb{Z}$.

In \cite{wild94} and \cite{wild95}, 
Wildberger has introduced a method to construct a hermitian finite hypergroup from a special graph by considering random walks on the graph. However, all graphs do not always produce hypergroups. 
A hypergroup is a $*$-algebra, which requires the product structure and its associativity. Note that the hypergroup derived from a graph becomes a commutative algebra. We may not be able to define a product of elements at all, 
and even if we can do the associativity or the commutativity may fail, depending on graphs.
In this paper, we treat only connected infinite graphs and connected finite graphs equipped with some condition which is weaker than the condition of self-centered treated in \cite{ikka-sawa19}.
These assumptions give the well-definedness of the product.
Hence it is important whether or not the product is commutative and associative to get hypergroups by this construction.
A non-associative hypergroup is called a pre-hypergroup.
In Wildberger's construction, a random walk on a graph is that a random walker jumps about distance which is not necessary the distance $1$, and we consider time evolutions of distances given by only two jumps from a fixed base point. 
Recently, Ikkai-Sawada \cite{ikka-sawa19} studied a condition of not necessarily finite graphs that one can get hermitian discrete hypergroups by Wildberger's construction. 
They showed that one can get hermitian discrete hypergroups from any distance regular graphs.

We shall give an outline of this paper. 
In Section \ref{Preliminary}, we will prepare the basic concept of graphs and recall Wildberger's construction.

We will explicitly define a probability space and random variables  which describe a time evolution of distances given by random walks on a Cayley graph in Section \ref{Time evolution of distances obtained from random walks}. We will also discuss the Markov property of the time evolution in the finite case.

In Section \ref{Hypergroup products}, we will find a probability theoretic interpretation of $m$-th products on the pre-hypergroup derived from a graph, related to random walks on the graph. For a graph $\Gamma$ which is not necessarily distance regular, let $H(\Gamma)$ be the pre-hypergroup with a basis $x_0,x_1,x_2,\ldots$. 
Then a product formed as $x_i\circ x_j$ has a probability theoretic meaning.
However, for $m>2$, the meaning of a product formed as $(( \cdots ((x_{i_1} \circ x_{i_2}) \circ x_{i_3}) \circ \cdots ) \circ x_{i_{m-1}})\circ  x_{i_m}$ has not been clarified yet.
We will show that the $m$-th product includes informations of probabilities with respect to $m$ times jumps when $\Gamma$ has  some symmetries which are weaker than the distance regularity.

We shall explain Section \ref{Realization of hypergroups as operator algebras}. 
Any discrete hermitian hypergroup is isomorphic to a matrix algebra whose basis is transition matrices. The matrix algebra associated with the hypergroup derived from a distance regular graph $\Gamma$ is closed to the Bose-Mesner algebra of $\Gamma$ in \cite{bose-mesn59}.
We can apply the construction of transition matrices to pre-hypergroup. 
We will estimate the operator norms of transition matrices obtained by the pre-hypergroups derived from a graph. When $\Gamma$ produces a hypergroup and has the symmetries the result of Section 4 implies that their matrix products describe a distribution of distances between a fixed vertex and a random vertex in each steps.

We prepare the notations used in this paper. 
Let $\mathbb{N}$ be the set of all positive integers, $\mathbb{N}_0=\mathbb{N}\cup\{0\}$, $\mathbb{Z}$ the set of all integers, $\mathbb{R}$ the set of all real numbers and $\mathbb{C}$ the set of all complex numbers. 
For a set $S$, 
let $\mathbb{C}S$ denote the free vector space of $S$ over the field $\mathbb{C}$.

\section{Preliminaries}\label{Preliminary}
In this section, we prepare definitions, notations and facts related with the graph theory and Wildberger's construction of hermitian (discrete) hypergroups from some graphs.
\subsection{Graphs}
We refer the reader to \cite{gods-royl01} for the general graph theory. 
Let $\Gamma$ is a graph with a vertex set $V$. 
When a base point $v_0\in
V$ is fixed the pair $(\Gamma,v_0)$ is called a pointed graph. For $v,w\in
V$, let $d(v,w)$ denote the distance between $v$ and $w$, that is, the length of the shortest paths from $v$ to $w$.
In particular, we denote by $| v | = d( v_0 , v )$ the distance between the base point $v_0$ and a vertex $v \in V$. We also define
\begin{align*}
&I(\Gamma) = I(\Gamma,v_0)=\{n\in\mathbb{N}_0\mid
| v |=n\ \mbox{for some }v\in
V\},\\
&M(\Gamma)=M(\Gamma,v_0)=\sup
I(\Gamma,v_0),\\
&S_{n}(w)=\{w'\in
V\mid
d(w,w')=n\}
\end{align*}
for each $n\in\mathbb{N}_0$ and $w\in
V$. In particular, the set $S_n(v_0)$ is sometimes denoted by $S_n$ simply. Note that $\sup_{n\in
I(\Gamma)}|S_n|<\infty$ if and only if $\sup_{n\in
I(\Gamma)}|S_n(v)|<\infty$ for all $v\in
V$. 
In this paper, assume that any graph $\Gamma$ is\\
~(i) simple, connected and locally finite,\\
~(ii) has at most countable vertices, and\\
~(iii) satisfies $S_{M(\Gamma)}(v) \neq \varnothing$ for any vertex $v$.\\
Here, the third condition in our assumption is weaker than self-centered.

We shall recall the notion of Cayley graphs and define some notations related with graphs. For the general theory of Cayley graphs, see \cite{kreb-shah11}. For a discrete group $G$ and a symmetric finite subset $S\subset
G$ not containing the unit element $e$ in $G$ and generating $G$, the Cayley graph ${\rm
Cay}(G,S)$ of the pair $(G,S)$ is a graph whose vertex set is $G$ and edges are defined by the follows: a vertex $v\in
G$ is adjacent to a vertex $w\in
G$ if $v^{-1}w\in
S$. The Cayley graph ${\rm
Cay}(G,S)$ is sometimes denoted by $G$ simply when we need not to specify the subset $S$. 
In this paper, we always assume any base point in a Cayley graph $G$ is the unit element $e$ in $G$.

Now, we shall introduce some symmetric conditions for graphs as follows:
%%%%%%%%%%%%%%%%%%%%%%%%%%%%%%%%%%%%%%%%%%%%%%%%%%%%%%%%%%%%%%%%%%%%%%%%%%%%%%
\begin{defi}
Let $\Gamma$ be a graph with a base point $v_0$. 
We define conditions
\begin{itemize}
\item[$(S1)$] the function $|S_i(\cdot)|$ is a constant for each $i\in
I(\Gamma,v_0)$, and
%we say $\Gamma$ is \underline{multi-regular}.
\item[$(S2)$] the function $| S_i(\cdot) \cap S_j(v_0) |$ is a constant on $S_k( v_0 )$ for each $i, j, k \in I(\Gamma)$.  
\end{itemize}
\end{defi}
%%%%%%%%%%%%%%%%%%%%%%%%%%%%%%%%%%%%%%%%%%%%%%%%%%%%%%%%%%%%%%%%%%%%%%%%%%%%%
\begin{defi}\label{d.r.g}
A graph $\Gamma$ with a vertex set $V$ is said to be distance regular if for every $i,j,k\in
\tilde{I}(\Gamma):=\{n\in\mathbb{N}_0\mid
d(v,w)=n\mbox{ for some }v,w\in
V\}$, the cardinality $|\{x\in
V\mid
d(v,x)=i,\ d(x,w)=j\}|$ is independent of the choice $v,w\in V$ with $d(v,w)=k$. 
The above cardinality is denoted by $Q(\Gamma)_{i,j}^k$.
\end{defi}
Note that the above definition of distance regular graphs is not the original definition but the one in the aspect of association schemes. 
It is known that they are equivalent (for example, see \cite[Proposition 2.4]{ikka-sawa19}). 
We refer the reader to \cite{bcn89} and \cite{dkt16} for the general theory of distance regular graphs. It is clear that any distance regular graph satisfies the conditions $(S1)$ and $(S2)$.
Of course, the condition $(S1)$ implies the assumption (iii).
%%%%%%%%%%%%%%%%%%%%%%%%%%%%%%%%%%%%%%%%%%%%%%%%%%%%%%%%%%%%%%%%%%%%%%%%%%%%%%%%
\begin{lemm}\label{cayleylemm}
Any Cayley graph $(G,S)$ {\rm(}with the base point $v_0=e${\rm)} satisfies the conditions $(S1)$ and
\begin{itemize}
\item[$(S3)$] $w\in S_i(v_0)$ if and only if $vw\in S_i(v)$ for all $v \in G$ and $i \in I(G)$.
\end{itemize}
\end{lemm}
\begin{proof}
For $ l \geq 2$, the vertices $v$ and $w\in G$ is called to have the $l$-length path 
if there exists a tuple $\mathbf{w}_l=(w_1, w_2, \ldots , w_{l-1}) \in G^{l-1}$ satisfies the condition $P_l(w_0, w)$ :
\[
\begin{cases}
0 \leq j \leq l -1 \Rightarrow w_j^{-1} w_{j +1} \in S, \\
0 \leq j \leq k \leq l,\ k - j \geq 2 \Rightarrow w_j^{-1} w_k \notin S, \\
w_j \in G\setminus \{w_0, w_1, \ldots, w_{j-1}, w_l \} \quad \text{for any $j = 1,2, \ldots, l -1$}.
\end{cases}
\]
where we put $w_0 = v$ and $w_l = w$ for convenience.
\par
Fix $i\in I(G)$ and $v \in G$. 
It suffices to show that the map $F_v : S_i(v_0) \ni w \mapsto v w  \in S_i(v)$ is bijective.
The case $i =1$ is trivial.
First we check the well-definedness of $F_v$
Take $w \in S_i(v_0)$.
Then there exists $\mathbf{w}_i =(w_1, w_2, \ldots , w_{i-1}) \in G^{i-1}$ such that $\mathbf{w}_i$ satisfies $P(v_0, w)$, and 
it holds that, for any $j = 1, 2, \ldots, i -2$ and $\mathbf{w}_j \in G^{j-1}$, 
$\mathbf{w}_j$ does not satisfies $P_{j}(v_0, w)$.
Note that these two conditions are equivalent to $w \in S_i(v_0)$.
Now we find that $v \mathbf{w}_j = (v w_1, v w_2, \ldots, v w_{i-1})$ satisfies $P_i(v, v w)$ 
and that  it holds that,
for any $j = 1, 2, \ldots, i -2$ and $\mathbf{w}_j \in G^{j-1}$, 
$\mathbf{w}_j$ does not satisfies $P_{j}(v, v w)$, 
which give the well-difinedness of $F_v$.
For the bijectivity, we can easily check that the map $F_{v^{-1}} : S_i(v) \ni u \mapsto v ^{-1}u  \in S_i(v_0)$ is also well-defined and that this map is inverse mapping of $F_v$.
\end{proof}
%%%%%%%%%%%%%%%%%%%%%%%%%%%%%%%%%%%%%%%%%%%%%%%%%%%%%%%%%%%%%%%%%%%%%%%%%%%%%%%
In general, the condition $(S2)$ is weaker than the distance regularity, however their symmetry conditions are equivalent for a Cayley graph as follows:
%%%%%%%%%%%%%%%%%%%%%%%%%%%%%%%%%%%%%%%%%%%%%%%%%%%%%%%%%%%%%%%%%%%%%%%%%%%%%%
\begin{lemm}\label{cayleylemm2}
Let $G$ be a Cayley graph with the base point $e\in
G$. If $G$ satisfies the condition $(S2)$ then it is distance regular.
\end{lemm}
\begin{proof}
Fix $i,j,k\in
\tilde{I}(G)$ and $g,g',h,h'\in
G$ with $d(g,h)=k=d(g',h')$. If we define the action $\varphi_x:G\ni
y\mapsto
xy\in
G$ for each $x\in
G$, by Lemma \ref{cayleylemm}, then we have $\varphi_{h^{-1}}(S_i(g)\cap
S_j(h))=S_i(h^{-1}g)\cap
S_j(e)$. By Lemma \ref{cayleylemm} again, we have $d(h^{-1}g,e)=d(g,h)=k=d(g',h')=d(h'^{-1}g',e)$, and hence the condition $(S2)$ implies that
\begin{align*}
&|S_i(g)\cap
S_j(h)|=|\varphi_{h^{-1}}(S_i(g)\cap
S_j(h))|=|S_i(h^{-1}g)\cap
S_j(e)|=|S_i(h'^{-1}g')\cap
S_j(e)|\\
&=|\varphi_{h'^{-1}}(S_i(g')\cap
S_j(h'))|=|S_i(g')\cap
S_j(h')|.
\end{align*}
It have been proved that $G$ is distance regular.
\end{proof}
%%%%%%%%%%%%%%%%%%%%%%%%%%%%%%%%%%%%%%%%%%%%%%%%%%%%%%%%%%%%%%%%%%%%%%%%%%%%
\begin{exam}
The $1$-dimensional lattice ${\rm Cay}(\mathbb{Z},\{\pm 1\})$ is distance regular, 
and the $2$-dimensional lattice ${\rm Cay}(\mathbb{Z}^2,\{ (\pm 1, 0), (0, \pm 1) \})$ satisfies the condition $(S1)$ and does not satisfy the condition $(S2)$.
\end{exam}
%%%%%%%%%%%%%%%%%%%%%%%%%%%%%%%%%%%%%%%%%%%%%%%%%%%%%%%%%%%%%%%%%%%%%%%%%%%%
\begin{exam}
For each $n\in\mathbb{N}$, the odd graph $\mathcal{O}_n$ with degree $n$ is a distance regular. It is known that $\mathcal{O}_n$ is not a Cayley graph if $n>2$ by \cite{gods80}.
\end{exam}
%%%%%%%%%%%%%%%%%%%%%%%%%%%%%%%%%%%%%%%%%%%%%%%%%%%%%%%%%%%%%%%%%%%%%%%%%%%%%%
\begin{exam}\label{prism}
The Cayley graph $\mathcal{P}_n={\rm
Cay}(\mathbb{Z}/n\mathbb{Z}\oplus\mathbb{Z}/2\mathbb,\{(\pm1,0),(0,1)\})$ called the $n$-gonal prism graph is distance regular if and only if $n=4$. For $n\neq4$, a pointed graph $\mathcal{P}_n$ with an arbitrary base point does not satisfy the condition $(S2)$.
\end{exam}
%%%%%%%%%%%%%%%%%%%%%%%%%%%%%%%%%%%%%%%%%%%%%%%%%%%%%%%%%%%%%%%%%%%%%%%%%%%%%%%%
\begin{exam}\label{tree}
The pair $(\mathcal{B},v_0)$ of the binary tree $\mathcal{B}$ and a base point $v_0$ in $\Gamma$ defined as {\rm
Figure} $1$, satisfies the condition $(S2)$ and does not satisfy the condition $(S1)$.
\[
\begin{xy}
(0,0)*{\circ}="A",
(-10,-6)*{\circ}="B",
(10,-6)*{\circ}="C",
(-15,-12)*{\circ}="D",
(-5,-12)*{\circ}="E",
(5,-12)*{\circ}="F",
(15,-12)*{\circ}="G",
{ "A" \ar @{-} "B" },
{ "A" \ar @{-} "C" },
{ "B" \ar @{-} "D" },
{ "B" \ar @{-} "E" },
{ "C" \ar @{-} "F" },
{ "C" \ar @{-} "G" },
(-15,-14)*{\vdots},
(-5,-14)*{\vdots},
(5,-14)*{\vdots},
(15,-14)*{\vdots},
(0,4)*{v_0},
(0,-22)*{\mbox{{\rm
Figure }}1}
\end{xy}
\]
\end{exam}
%%%%%%%%%%%%%%%%%%%%%%%%%%%%%%%%%%%%%%%%%%%%%%%%%%%%%%%%%%%%%%%%%%%%%%%%%%%%
\subsection{Hypergroups derived from graphs}
We shall recall Wildberger's construction of hypergroups from some graphs. We refer the reader to \cite{wild95} and \cite{ikka-sawa19} for details of the theory.

Let $\Gamma$ be a graph with a vertex set $V$ and $v_0\in
V$. 
Let $H(\Gamma,v_0)=\{x_i\}_{i\in
I(\Gamma,v_0)}$ with dummy symbols $x_i$ for $i\in
I(\Gamma,v_0)$. If $\Gamma$ is an infinite graph, then $I(\Gamma,v_0)=\mathbb{N}_0$. 
For $i,j,k\in I(\Gamma,v_0)$, we define
\begin{align}
p_{i,j}^{k}=\frac{1}{|S_{i} (v_0)|} \sum_{v \in S_{i} (v_0)}\frac{|S_{j}(v)\cap S_{k} (v_0)|}{|S_{j} (v)|}.\label{pijk}
\end{align}
The set $\{p_{i,j}^k\}_{k\in
I(\Gamma,v_0)}$ is the distribution of distances between the base point $v_0$ and a random vertex $w\in S_{j} (v)$ after a jump to a random vertex $v\in S_{i}(v_0)$. 
By our assumptions for graphs, 
the probability $p_{i,j}^{k}$ is well-defined for any $i,j,k\in
I(\Gamma,v_0)$ by \cite[Proposition 3.1]{ikka-sawa19}. 
Also, we define a product on the free vector space $\mathbb{C}H(\Gamma,v_0)$ by
\begin{align}
x_i\circ
x_j=\sum_{k\in
I(\Gamma,v_0)}p_{i,j}^kx_k.
\label{defiproduct}
\end{align}
for each $x_i,x_j\in
H(\Gamma,v_0)$. 
Note that $|\{k\in
I(\Gamma,v_0)\mid
p_{i,j}^k\neq0\}|<\infty$ for all $i,j\in
I(\Gamma,v_0)$ even if $\Gamma$ is infinite.

When the graph $\Gamma$ has a ``good symmetry'', defining $x_i^*=x_i$ for each $i\in
I(\Gamma)$, the triple $(H(\Gamma,v_0),\circ,*)$ forms a discrete hermitian  hypergroup in the following sense.
%%%%%%%%%%%%%%%%%%%%%%%%%%%%%%%%%%%%%%%%%%%%%%%%%%%%%%%%%%%%%%%%%%%%%%%%%%%
\begin{defi}
Let $H=\{x_i\}_{i\in
I(H)}$ be a countable set whose elements are parametrized by a index set $I(H)=\{0,\ldots,N\}$ for some $N\in\mathbb{N}_0$ or $\mathbb{N}_0$. 
Suppose $\circ$ and $*$ are a binary operation and an involution, respectively, on the free vector space $\mathbb{C}H$.
We give the following three definitions :\\
$\bullet$ The triple $(H,\circ,*)$ is called a discrete hypergroup if the following conditions are satisfied.
\begin{enumerate}
\item
The triple $(\mathbb{C}H,\circ,*)$ is a $*$-algebra with the unit $x_0\in
H$.
\item
For $i,j\in I(H)$, if $x_i\circ x_j=\sum_{k=0}^mp_{i,j}^kx_k$ and $p_{i,j}^k\in\mathbb{C}\ (k=0,\ldots,m)$, then we have $p_{i,j}^k\geq0$ for all $k=0,\ldots,m$ and $\sum_{k=0}^mp_{i,j}^k=1$.
\item For all $i,j\in I(H)$, 
one has $p_{i,j}^0\neq0$ if and only if $x_i=x_j^*$.
\item The restriction $*|_H$ maps onto $H$.
\end{enumerate}
We call the above $\{p_{i,j}^k\}_{i,j,k\in I(H)}$ the structure constants of $H$.\\
$\bullet$ The hypergroup $(H,\circ,*)$ is said to be 
\begin{itemize}
\item[-] finite if $I(H)$ is finite,
\item[-] commutative if $(\mathbb{C}H, \circ)$ is a commutative algebra, 
\item[-] hermitian if the restriction $*|_H$ is the identity map.
\end{itemize}
$\bullet$ In this paper, the pair $(H,\circ)$ is called a pre-hypergroup if $(\mathbb{C}H,\circ)$ is an algebra, which may fail the associativity and satisfies $(2)$ and
\begin{enumerate}[]
\item
\ \ \ (3')\  For all $i,j\in I(H)$, 
one has $p_{i,j}^0\neq0$ if and only if $x_i=x_j$.
\end{enumerate}
In this paper, a hypergroup $(H,\circ,*)$ or pre-hypergroup $(H,\circ)$ is denoted by $H$ simply. 
\end{defi}

We refer the reader to \cite{lass05} for details of the general theory of discrete commutative hypergroups. 
Note that a discrete hermitian hypergroup is automatically commutative.
We only treat discrete hermitian hypergroups and pre-hypergroups. Now, we note that, 
for any graph $\Gamma$ with a base point,
one can get a pre-hypergroup $H(\Gamma)$,
and it becomes a hypergroup if and only if
it holds that the conditions of the associativity and the commutativity:
\[
\sum_{l\in I(\Gamma)}p_{h,i}^lp_{l,j}^k
=\sum_{l\in I(\Gamma)}p_{i,j}^lp_{h,l}^k,\ p_{i,j}^k=p_{j,i}^k
\]
for all $h,i,j,k\in
I(\Gamma)$.
Also, any hypergroups derived from infinite graphs are polynomial hypergroups in the sense of \cite{lass83}. 

\begin{rema}
Our construction of hypergroups is slightly extension of the one refered in \cite{ikka-sawa19} in which we treat only self-centered graphs. Thanks to our construction, we can sometimes get a hypergroup from a graph equipped with few symmetries as {\rm
Example \ref{allnotexample}}.
\end{rema}

All graphs do not always produce hypergroups, however we have a sufficient condition of graphs for producing hypergroups as follows:
%%%%%%%%%%%%%%%%%%%%%%%%%%%%%%%%%%%%%%%%%%%%%%%%%%%%%%%%%%%%%%%%%%%%%%%%%%%%%%%%
\begin{theo}{\rm(\cite[Theorem 3.3]{ikka-sawa19})}\label{ikka-sawadr}
If $(\Gamma,v_0)$ is a pointed distance regular graph, then $H(\Gamma,v_0)$ is a hermitian discrete hypergroup with the structure constants $\{p_{i,j}^k\}_{i,j,k\in
I(\Gamma,v_0)}$. Moreover, the structure is independent of the choice of $v_0$.
\end{theo}
%%%%%%%%%%%%%%%%%%%%%%%%%%%%%%%%%%%%%%%%%%%%%%%%%%%%%%%%%%%%%%%%%%%%%%%%%%%%%%%%
%%%%%%%%%%%%%%%%%%%%%%%%%%%%%%%%%%%%%%%%%%%%%%%%%%%%%%%%%%%%%%%%%%%%%%%%%%%%%%%%
\begin{exam}{\rm(\cite[Corollary 3.8.]{ikka-sawa19})}\label{freegroup}
A typical example of graphs producing a hypergroup is the $1$-dimensional lattice ${\rm
Cay}(\mathbb{Z},\{\pm1\})$ and the hypergroup $H(\mathbb{Z})$ has the structure given by $x_i\circ
x_j=\frac{1}{2}x_{|i-j|}+\frac{1}{2}x_{i+j}$ for each $i,j\in\mathbb{N}_0$. The hypergroup $H(\mathbb{Z})$ is the polynomial hypergroup with respect to the Chebyshev polynomials.

In general, let $F_n$ be the $n$-free group with the generator $A =\{ a_1, a_2, \ldots, a_n \}$.
The Cayley graph ${\rm Cay}(F_n,A\cup A^{-1})$ is distance regular, and hence it produces a hypergroup
\end{exam}
%%%%%%%%%%%%%%%%%%%%%%%%%%%%%%%%%%%%%%%%%%%%%%%%%%%%%%%%%%%%%%%%%%%%%%%%%%%%%
%%%%%%%%%%%%%%%%%%%%%%%%%%%%%%%%%%%%%%%%%%%%%%%%%%%%%%%%%%%%%%%%%%%%%%%%%%%%
\begin{exam}
The $2$-dimensional lattice$\mathbb{Z}^2$ does not produce a hypergroup. Indeed, we can check that $(x_1\circ
x_2)\circ
x_3\neq
x_1\circ(x_2\circ
x_3)$. The binary tree $\mathcal{B}$ with the base point $v_0$ defined in {\rm
Example} \ref{tree} does not produce a hypergroup. Indeed, the commutativity fails.
\end{exam}
%%%%%%%%%%%%%%%%%%%%%%%%%%%%%%%%%%%%%%%%%%%%%%%%%%%%%%%%%%%%%%%%%%%%%%%%%%%%%%%
Note that the distance regularity and the conditions $(S1),(S2)$ are not necessary conditions for producing hypergroups as the following examples.
%%%%%%%%%%%%%%%%%%%%%%%%%%%%%%%%%%%%%%%%%%%%%%%%%%%%%%%%%%%%%%%%%%%%%%%%%%%%%%
%%%%%%%%%%%%%%%%%%%%%%%%%%%%%%%%%%%%%%%%%%%%%%%%%%%%%%%%%%%%%%%%%%%%%%%%%%%%%%%
\begin{exam}{\rm(\cite[Subsection 4.2]{ikka-sawa19})}\label{someg}
\begin{enumerate}[{\rm(1)}]
\item
For $n\neq
4$, the $n$-gonal prism graph $\mathcal{P}_n$ in {\rm
Example \ref{prism}} is not distance regular, 
however it produces a hypergroup.
\item
The complete bipartite graph $\mathcal{K}_{n,m}$ with partitions $(n,m)$ produces a hypergroup for each $n,m\in\mathbb{N}$. For example, $\mathcal{K}_{2,3}$ is drown as {\rm
Figure 2}. The graph $\mathcal{K}_{n,m}$ is not regular if and only if $n\neq
m$ and then a pointed graph $(\mathcal{K}_{n,m},v_0)$ with an arbitrary base point $v_0$ satisfies the condition $(S2)$.
\item
An important example of graph producing a hypergroup is one drown as {\rm
Figure $3$} which is not distance regular and satisfies the condition $(S1)$. This graph $\Gamma$ produces different hypergroups depends on base points $w_0,w_0'$. Note that $(\Gamma,w_0)$ satisfies the condition $(S2)$ and $(\Gamma,w_0')$ does not satisfy it.
\end{enumerate}
\end{exam}
\begin{exam}\label{allnotexample}
A pointed graph $(\Gamma,u_0)$ drawn as {\rm
Figure $4$} is not a Cayley graph, does not satisfy $(S1)$ and $(S2)$,
and produces a hypergroup.
\end{exam}
\[
\begin{xy}
(0,0)*{\circ}="A",
(20,5)*{\circ}="B",
(0,10)*{\circ}="C",
(20,-5)*{\circ}="D",
(20,15)*{\circ}="E",
(60,15)*{\circ}="A'",
(50.5,8.5)*{\circ}="B'",
(69.5,8.5)*{\circ}="C'",
(47.5,0)*{\circ}="D'",
(72.5,0)*{\circ}="E'",
(54.5,-6)*{\circ}="F'",
(65.5,-6)*{\circ}="G'",
(110,15)*{\circ}="A1",
(110,5)*{\circ}="B1",
(100,10)*{\circ}="C1",
(120,10)*{\circ}="D1",
{ "A1" \ar @{-} "B1" },
{ "A1" \ar @{-} "C1" },
{ "A1" \ar @{-} "D1" },
{ "C1" \ar @{-} "B1" },
{ "D1" \ar @{-} "B1" },
{ "A" \ar @{-} "B" },
{ "A" \ar @{-} "D" },
{ "A" \ar @{-} "E" },
{ "C" \ar @{-} "B" },
{ "C" \ar @{-} "D" },
{ "C" \ar @{-} "E" },
{ "A'" \ar @{-} "B'" },
{ "A'" \ar @{-} "C'" },
{ "A'" \ar @{-} "F'" },
{ "A'" \ar @{-} "G'" },
{ "C'" \ar @{-} "B'" },
{ "D'" \ar @{-} "B'" },
{ "B'" \ar @{-} "E'" },
{ "D'" \ar @{-} "F'" },
{ "G'" \ar @{-} "F'" },
{ "C'" \ar @{-} "D'" },
{ "C'" \ar @{-} "E'" },
{ "E'" \ar @{-} "G'" },
{ "E'" \ar @{-} "F'" },
{ "D'" \ar @{-} "G'" },
(12,-12)*{\mbox{{\rm
Figure }}2},
(55,17)*{w_0},
(50,-9)*{w_0'},
(60,-16)*{\mbox{{\rm
Figure }}3},
(105,17)*{u_0},
(110,-1)*{\mbox{{\rm
Figure }}4}
\end{xy}
\]
%%%%%%%%%%%%%%%%%%%%%%%%%%%%%%%%%%%%%%%%%%%%%%%%%%%%%%%%%%%%%%%%%%%%%%%%%%%%%%%%
%%%%%%%%%%%%%%%%%%%%%%%%%%%%%%%%%%%%%%%%%%%%%
%%%%%%%%%%%%%%%%%%%%%%%%%%%%%%%%%%%%%%%%%%%%%%%%%%%%%%%%%%%%%%%%%%%%%%%%%%%%%
\begin{rema}
Let $\Gamma,\Gamma'$ be two distance regular graphs. If the two hypergroups $H(\Gamma),H(\Gamma')$ are isomorphic, that is, there is a bijective $*$-homomorphism $\Phi:\mathbb{C}H(\Gamma)\to\mathbb{C}H(\Gamma')$, then each constants $Q(\Gamma)_{i,j}^k$ associated with $\Gamma$ in Definition \ref{d.r.g}, coincide with $Q(\Gamma')_{i,j}^k$. Indeed, the hypergroup structures are represented as $p_{i,j}^k=\frac{Q(\Gamma)_{j,k}^i}{Q(\Gamma)_{j,j}^0}$ and $Q(\Gamma)_{j,0}^j=1$.
\end{rema}

\section{Distances distribution obtained from random walks on Cayley graphs}\label{Time evolution of distances obtained from random walks}
We consider a random walk on a Cayley graph. For a convenience in the later sections, we shall give precise definitions a probability measure and random variables describing time evolutions of distances between the base point and vertices which a random walker passes trough, as follows:
%%%%%%%%%%%%%%%%%%%%%%%%%%%%%%%%%%%%%%%%%%%%%%%%%%%%%%%%%%%%%%%%%%%%%%%%%%%%%%%
\begin{defi}\label{defprobcay}
Let $G$ be a Cayley graph with the base point $v_0 = e$ and $\{\alpha_i\}_{i\in I(G,v_0)}$ a sequence of non-negative numbers with $\sum_{i\in
I(G,v_0)}\alpha_i|S_i(v_0)|=1$. We define a probability measure $\mathbb{P}_0$ on $G$ by $\mathbb{P}_0(\{v\})=\alpha_{|v|}$ for each $v\in
G$ and denote by $\mathbb{P}$ the probability measure on $\Omega=G^\mathbb{N}$ via Kolmogorov's extension theorem. For each $m\in\mathbb{N}$, we define a $G$-valued random variable $X_m$ on $\Omega$, which describes a distance of the $n$-th jump, by $X_m((\omega_n)_{n=1}^\infty)=\omega_m$ for $(\omega_n)_{n=1}^\infty\in\Omega$.

We call $\{\alpha_i\}_{i\in
I(G,v_0)}$ a distribution of $G$. Also, if $G$ is finite and $\alpha_i=\alpha_j$ for all $i,j\in
I(G,v_0)$, we say that $G$ has the uniform distribution.

For each $n\in\mathbb{N}$, we also define $\mathbb{N}_0$-valued random variables $Z_n$ on $\Omega$, which describes a distance between the unit element and a random vertex at the time $n$, by $Z_n=|X_1X_2\cdots
X_n|.$ Since we assume that a random walker leaves from the unit element, suppose $Z_0=0$.
\end{defi}
For $n\in\mathbb{N}$, we have $\mathbb{P}(X_n=v)=\alpha_{|v|}$ and $\mathbb{P}(|X_n|=i)=\alpha_i|S_i(v_0)|$ for any $v\in
G$ and $i\in
I(G)$.

We shall discuss the Markov property of the process $\{Z_n\}_{n=0}^\infty$ defined in Definition \ref{defprobcay} and its stationary distribution. We refer the reader to \cite{brem99} for the general theory of Markov chains.

Let $G$ be a not necessarily finite Cayley graph with a distribution $\{\alpha_n\}_{n=0}^\infty$ satisfying $\alpha_n>0$ for all $n\in
I(G)$. Suppose $\{X_n\}_{n=1}^\infty$ and $\{Z_n\}_{n=0}^\infty$ are the processes defined in Definition \ref{defprobcay}. 
By the condition $(S3)$, we can calculate as 
\begin{align*}
&\mathbb{P}(Z_0=i_0,\ Z_1=i_1,\ \ldots,\ Z_k=i_k)\\
&=\sum_{w_1\in
S_{i_1}(v_0)}\sum_{w_2\in
S_{i_2}(w_1^{-1})}\sum_{w_3\in
S_{i_3}(w_2^{-1}w_1^{-1})}\cdots\sum_{w_k\in
S_{i_k}(w_{k-1}^{-1}\cdots
w_1^{-1})}\alpha_{|w_1|}\cdots\alpha_{|w_k|}
\end{align*}
for each $i_1,\ldots,i_k$, where suppose $i_0=0$. 
Hence, in general, the conditional probability
\[
\mathbb{P}\left(Z_{n+1}=i_{n+1} ~\bigg{|}~
Z_0=i_0,\ Z_1=i_1,\ \ldots,\ Z_{n-1}=i_{n-1},\ Z_n=i_n\right)
\]
depends on $i_1,\ldots,i_{n-1}\in
I(G)$, that is, the process $\{Z_n\}_{n=0}^\infty$ is not always a Markov chain. However, if $G$ is finite and has the uniform distribution, that is, $\alpha_n=\frac{1}{|G|}$ for all $i \in I(G)$,
then we can show that $\{Z_n\}_{n=0}^\infty$ is a Markov chain and, moreover,
has the independently identically distribution as follows: it holds that
\begin{equation}\label{probuniform}
\mathbb{P}(Z_0=i_0,\ Z_1=i_1,\ \ldots,\ Z_k=i_k)=\frac{1}{|G|^k}\prod_{l=1}^k|S_{i_l}|
\end{equation}
by the condition $(S1)$. Note that $\mathbb{P}(Z_0=i_0,\ldots,Z_n=i_n)\neq0$ because $S_{l}(v_0)\neq\varnothing$ for all $l\in
I(G)$. 
We have $\mathbb{P}\left(Z_{n+1}=i_{n+1} ~\bigg{|}~
Z_0=i_0,\ldots,Z_n=i_n\right)
=\frac{|S_{i_{n+1}}|}{|G|}$ which is independent on the choice of $i_1,\cdots,i_{n-1},i_n$ and coincide with $\mathbb{P}(Z_n = i_{n+1})$.

In such a situation, we denote by $P=(p_{ij})_{i,j\in
I(G)}$ the transition probability matrix associated with the ergodic Markov chain $\{Z_n\}_{n=0}^\infty$. Then $p_{ij}=\frac{|S_j|}{|G|}$ depends only $j$, and hence we have
\[
P=\frac{1}{|G|}\begin{pmatrix}1&|S_1|&|S_2|&\cdots&|S_{M(G)}|
\\1&|S_1|&|S_2|&\cdots&|S_{M(G)}|\\
\vdots&\vdots&\vdots&\cdots&\vdots\\
1&|S_1|&|S_2|&\cdots&|S_{M(G)}|\end{pmatrix},
\]
which is idempotent, with a stationary distribution
\begin{equation}\label{s.d.}
\pi_G=\frac{1}{|G|}(1,|S_1|,|S_2|,\ldots,|S_{M(G)}|).
\end{equation}

\section{Hypergroup products}\label{Hypergroup products}
In this section, we will clarify a relation between time evolutions of distances between a base point $v_0$ on a graph $\Gamma$ and random vertices, and products on the pre-hypergroup $H(\Gamma)=\{x_i\}_{i\in
I(\Gamma)}$ derived from $\Gamma$. In other words, our goal is to give a probability theoretic interpretation to $m$-th products on $H(\Gamma)$. Of course, by the construction, a product formed as $x_i\circ
x_j$ has the probability theoretic meaning. However, the meaning of a product formed as $(( \cdots ((x_{i_1} \circ x_{i_2}) \circ x_{i_3}) \circ \cdots ) \circ x_{i_{m-1}})\circ  x_{i_m}$ has not been clarified for $m>2$.

First, the familiy $\{p_{i,j}^k\}$, 
given by the hypergroup product (\ref{pijk}) with respect to a Cayley graph $G$, 
can be represented by conditional probabilities as follows:
%%%%%%%%%%%%%%%%%%%%%%%%%%%%%%%%%%%%%%%%%%%%%%%%%%%%%%%%%%%%%%%%%%%%%%%%%%%%%
\begin{prop}\label{relationCayley}
Let $G$ be a Cayley graph. For every $i,j,k\in
I(G)$ with $\alpha_i\neq0$ and $\alpha_j\neq0$, we have $p_{i,j}^k=\mathbb{P}\left(Z_2=k ~\bigg{|}~
|X_1|=i,\ |X_2|=j\right).$
\end{prop}
\begin{proof}
Recall that $G$ satisfies the conditions $(S1)$ and $(S3)$ by Lemma \ref{cayleylemm}. The condition $(S3)$ implies that there is a bijection between $\{w\in S_j(v_0)\mid |vw|=k\}$ and $S_j(v)\cap S_k(v_0)$ for all $v\in S_i(v_0)$. 
Thus, we have
\begin{align*}
&\mathbb{P}(|X_1X_2|=k,\ |X_1|=i,\ |X_2|=j)\\
&=\sum\left\{\mathbb{P}(X_1=v,\ X_2=w) \mid (v,w)\in S_i(v_0)\times S_j(v_0),\ vw\in S_k(v_0)\right\}\\
&=\sum\left\{\alpha_i\alpha_j \mid(v,w) \in S_i(v_0)\times S_j(v_0),\ vw\in S_k(v_0)\right\}\\
&=\alpha_i\alpha_j\sum_{v\in S_i(v_0)}\sum\left\{1\mid
w\in S_j(v_0),\ vw\in S_k(v_0)\right\}\\
&=\alpha_i\alpha_j\sum_{v\in S_i(v_0)}\sum\left\{1\mid w\in S_j(v)\cap S_k(v_0)\right\}
=\alpha_i\alpha_j\sum_{v\in S_i(v_0)}| S_j(v)\cap S_k(v_0)|.
\end{align*}
Since $\mathbb{P}(|X_1|=i,\ |X_2|=j)=\alpha_i\alpha_j| S_i(v_0)|  | S_j(v_0)|$,  
the condition $(S1)$ implies that
\begin{align*}
&\mathbb{P}\left(|X_1X_2|=k ~\bigg{|}~
|X_1|=i,\ |X_2|=j\right)=\frac{1}{| S_i(v_0)|}\sum_{v\in S_i(v_0)}\frac{| S_j(v)\cap S_k(v_0)|}{| S_j(v_0)|}\\
&=\frac{1}{| S_i(v_0)|}\sum_{v\in S_i(v_0)}\frac{| S_j(v)\cap S_k(v_0)|}{| S_j(v)|}=p_{i,j}^k.
\end{align*}
\end{proof}
%%%%%%%%%%%%%%%%%%%%%%%%%%%%%%%%%%%%%%%%%%%%%%%%%%%%%%%%%%%%%%%%%%%%%%%%%%%%%%%%
Note that the distribution $\{\alpha_n\}$ does not appear in a conditional probability as the form in the previous theorm.
Now, we shall define three type objects derived from $m$ jumps from a base point as follows:
%%%%%%%%%%%%%%%%%%%%%%%%%%%%%%%%%%%%%%%%%%%%%%%%%%%%%%%%%%%%%%%%%%%%%%%%%%%%%%%%%%%%%%%%%%%%%%%%%%%%%%%%%%%%%
\begin{defi}\label{probs}
Let $(\Gamma, v_0)$ be a pointed graph with a vertex set $V$ and $G$ a {\rm(}pointed{\rm)} Cayley graph. 
\begin{itemize}
\item[(1)] For $i_1, i_2, \ldots , i_m \in I(\Gamma)$, 
we define
\begin{align}
& PL(i_1,i_2,\ldots,i_m) = (( \cdots ((x_{i_1} \circ x_{i_2}) \circ x_{i_3}) \circ \cdots ) \circ x_{i_{m-1}})\circ  x_{i_m}, \label{eqn:ALS}\\
& J( i_1, i_2, \ldots , i_m ) 
= \sum_{v_1 \in S_{i_1}(v_0)}\sum_{v_2 \in S_{i_2}(v_1)} \cdots \sum_{v_m \in S_{i_m}(v_{m -1})}
\frac{1}{\prod_{j =1}^{m}| S_{i_j} ( v_{j -1} ) |} x_{| v_m |}.\label{eqn:EL}
\end{align}
\item[(2)] Let $\mathbb{P}$ be the probability measure given in {\rm
Definition \ref{defprobcay}} with respect to a sequence $\{\alpha_i\}_{i\in
I(G,v_0)}$ of positive numbers. For $i_1, i_2, \ldots , i_m \in I(G)$ and $k \in I (G)$, 
we define
\begin{equation}\label{eqn:prob}
p_{i_1, i_2, \ldots , i_m}^k
=\mathbb{P}\left( Z_m = k ~\bigg{|}~ | X_1 | = i_1 , | X_2 | = i_2 , \ldots, | X_m | = i_m\right).
\end{equation}
\end{itemize}
\end{defi}
%%%%%%%%%%%%%%%%%%%%%%%%%%%%%%%%%%%%%%%%%%%%%%%%%%%%%%%%%%%%%%%%%%%%%%%%%%%%%%%%%%%%%%%%%%%%%%%%%%%%%%%%%%%%%
\begin{rema}\label{probrema}
\begin{enumerate}[{\rm(1)}]
\item
The right hand side of $\eqref{eqn:ALS}$ means the $m-1$ times products from the left step by step. 
\item
The coefficient $\tilde{p}_{i_1,\ldots,i_m}^k$ of $x_k$ in $J(i_1,\ldots,i_m)$ is the probability that a random walker reaches a vertex whose distance from the base point $v_0$ is $k$ under $m$-steps jumps as $v_0\xrightarrow{i_1}\cdot\xrightarrow{i_2}\cdots\xrightarrow{i_m}\cdot$.
\item
By {\rm
Proposition \ref{relationCayley}}, the probability $p_{i,j}^k$ in \eqref{eqn:prob} is well-defined for each $i,j\in
I(G)$.
\end{enumerate}
\end{rema}
%%%%%%%%%%%%%%%%%%%%%%%%%%%%%%%%%%%%%%%%%%%%%%%%%%%%%%%%%%%%%%%%%%%%%%%%%%%%%%%

The definition \eqref{eqn:EL} will play a role of giving a probability theoretic interpretation to $m$-th products on $H(\Gamma)$ for a graph which is not necessary a Cayley graph. For a Cayley graph, it is shown that \eqref{eqn:prob} coincides with the coefficient of $x_k$ in \eqref{eqn:EL} as the following theorem.

\begin{theo}\label{2-3}
Let $G$ be a Cayley graph. For all $m\geq2$ and $i_1,\ldots,i_m,k\in
I(G)$, we have
\begin{align}
&\sum_{v_1 \in S_{i_1}(v_0)}\sum_{v_2 \in S_{i_2}(v_1)} \cdots \sum_{v_{m-1} \in S_{i_{m-1}}(v_{m -2})}|S_{i_m}(v_{m-1})\cap
S_k(v_0)|\label{cardset}\\
&=|\{(v_1,\ldots,v_m)\in
\prod_{k=1}^m
S_{i_k}(v_0)\mid|v_1\cdots
v_m|=k\}|\nonumber,
\end{align}
and $\tilde{p}_{i_1,\ldots,i_m}^k=p_{i_1,\ldots,i_m}^k$ holds.
\end{theo}
\begin{proof}
We shall prove \eqref{cardset}. As the proof of Proposition \ref{relationCayley}, there exists a bijection between $S_{i_m}(v_{m-1})\cap
S_k(v_0)$ and $\{v_m\in
S_{i_m}(v_0)\mid|v_{m-1}v_m|=k\}$, and hence we have
\begin{align*}
&\sum_{v_1 \in S_{i_1}(v_0)} \cdots \sum_{v_{m-2} \in S_{i_{m-2}}(v_{m -3})}\sum_{v_{m-1} \in S_{i_{m-1}}(v_{m -2})}|S_{i_m}(v_{m-1})\cap
S_k(v_0)|\\
&=\sum_{v_1 \in S_{i_1}(v_0)} \cdots \sum_{v_{m-2} \in S_{i_{m-2}}(v_{m -3})}\sum_{v_{m-1} \in S_{i_{m-1}}(v_{m -2})}|\{v_m\in
S_{i_m}(v_0)\mid|v_{m-1}v_m|=k\}|\\
&=\sum_{v_1 \in S_{i_1}(v_0)} \cdots \sum_{v_{m-2} \in S_{i_{m-2}}(v_{m -3})}|\{(v_{m-1},v_m)\in
S_{i_{m-1}}(v_{m-2})\times
S_{i_m}(v_0)\mid|v_{m-1}v_m|=k\}|\\
&=\sum_{v_1 \in S_{i_1}(v_0)} \cdots \sum_{v_{m-2} \in S_{i_{m-2}}(v_{m -3})}|\{(v_{m-1},v_m)\in
S_{i_{m-1}}(v_0)\times
S_{i_m}(v_0)\mid|v_{m-2}v_{m-1}v_m|=k\}|
\end{align*}
Repeating this argument, it equals to $|\{(v_1,\ldots,v_m)\in\prod_{k=1}^m
S_{i_k}(v_0)\mid|v_1\cdots
v_m|=k\}|.$

By \eqref{cardset} and $(S1)$, we have
\begin{align*}
& \tilde{p}_{i_1,\ldots,i_m}^k=\frac{|\{(v_1,\ldots,v_m)\in
\prod_{k=1}^m
S_{i_k}(v_0)\mid|v_1\cdots
v_m|=k\}|}{\prod_{j=1}^m|S_{i_j}(v_{0})|}
=p_{i_1,\ldots,i_m}^k.
\end{align*}
This completes the proof.
\end{proof}

Next, we shall consider the definitions \eqref{eqn:ALS} and \eqref{eqn:EL}. 
In the following theorem, it will be shown that $PL(i_1,\ldots,i_m)
= J(i_1,\ldots,i_m)$ for a pointed graph $(\Gamma,v_0)$ with the conditions (S1) and (S2). 
However, the case for two jumps $(m=2)$ can be shown without the both of the conditions (S1) and (S2) as follows: 
for a pointed graph $(\Gamma,v_0)$, and $i,j\in I(\Gamma,v_0)$, 
by the definitions, we have
\begin{align}
& J(
i,j)=\sum_{k\in
I(\Gamma)}\sum_{v_1\in
S_i(v_0)}\sum_{v_2\in
S_j(v_1)\cap
S_k(v_0)}\frac{1}{|S_i(v_0)||S_j(v_1)|}x_k \label{i,j}\\
&=\sum_{k\in
I(\Gamma)}\frac{1}{|S_i(v_0)|}\sum_{v_1\in
S_i(v_0)}\frac{|S_j(v_1)\cap
S_k(v_0)|}{|S_j(v_1)|}x_k=\sum_{k\in
I(\Gamma)}p_{i,j}^kx_k= PL(i,j).\nonumber
\end{align} 

%%%%%%%%%%%%%%%%%%%%%%%%%%%%%%%%%%%%%%%%%%%%%%%%%%%%%%%%%%%%%%%%%%%%%%%%%%%%%%

\begin{theo}\label{1-2}
Suppose that $(\Gamma, v_0)$ is a pointed graph equipped with the conditions $(S1)$ and $(S2)$. 
Then we have $PL(i_1,i_2,\ldots,i_m) =  J( i_1, i_2, \ldots , i_m )$ for $i_1, i_2, \ldots, i_m \in I(\Gamma)$.
\end{theo}
\begin{proof}
We shall show the theorem by induction. Assume that $PL(i_1,i_2,\ldots,i_{m-1}) =  J( i_1, i_2, \ldots , i_{m-1} ).$ As the previous argument $ J(
i,j)= PL(i,j)$ in (\ref{i,j}), we have
\begin{align*}
& PL(i_1,\ldots,i_m)=(x_{i_1}\circ\cdots\circ
x_{i_{m-1}})\circ
x_{i_m}= J(
i_1,\ldots,i_{m-1})\circ
x_{i_m}\\
&=\left(\sum_{v_1 \in S_{i_1}(v_0)} \cdots \sum_{v_{m-1} \in S_{i_{m-1}}(v_{m -2})}
\frac{1}{\prod_{j =1}^{m-1}| S_{i_j} ( v_{j -1} ) |} x_{| v_{m-1} |}\right)\circ
x_{i_m}\\
&=\sum_{v_1 \in S_{i_1}(v_0)} \cdots \sum_{v_{m-2} \in S_{i_{m-2}}(v_{m-3})}\sum_{k\in
I(\Gamma)}\sum_{v_{m-1} \in S_{i_{m-1}}(v_{m -2})\cap
S_k(v_0)}
\frac{1}{\prod_{j =1}^{m-1}| S_{i_j} ( v_{j -1} ) |} x_{k}\circ
x_{i_m}\\
&=\sum_{v_1 \in S_{i_1}(v_0)} \cdots \sum_{v_{m-2} \in S_{i_{m-2}}(v_{m-3})}\sum_{k\in
I(\Gamma)}\frac{|S_{i_{m-1}}(v_{m -2})\cap
S_k(v_0)|}{\prod_{j =1}^{m-1}| S_{i_j} ( v_{j -1} ) |} x_{k}\circ
x_{i_m}.
\end{align*}
Now, by $(S1)$ and $(S2)$, we have
\begin{align*}
&\sum_{k\in
I(\Gamma)}\frac{|S_{i_{m-1}}(v_{m -2})\cap
S_k(v_0)|}{\prod_{j =1}^{m-1}| S_{i_j} ( v_{j -1} ) |} x_{k}\circ
x_{i_m}\\
&=\sum_{k\in
I(\Gamma)}\frac{|S_{i_{m-1}}(v_{m -2})\cap
S_k(v_0)|}{\prod_{j =1}^{m-1}| S_{i_j} ( v_{j -1} ) |} \sum_{l\in
I(\Gamma)}\frac{1}{|S_k(v_0)|}\sum_{v\in
S_k(v_0)}\frac{|S_{i_m}(v)\cap
S_l(v_0)|}{|S_{i_m}(v)|}x_l\\
&=\sum_{l\in
I(\Gamma)}\sum_{k\in
I(\Gamma)}\frac{|S_{i_{m-1}}(v_{m -2})\cap
S_k(v_0)|}{\prod_{j =1}^{m}| S_{i_j} ( v_{j -1} ) |}\frac{1}{|S_k(v_0)|}\sum_{v\in
S_k(v_0)}|S_{i_m}(v)\cap
S_l(v_0)|x_l\\
&=\sum_{l\in
I(\Gamma)}\sum_{k\in
I(\Gamma)}\sum_{v_{m-1}\in
S_{i_{m-1}}(v_{m-2})\cap
S_k(v_0)}\sum_{v\in
S_k(v_0)}\frac{1}{|S_k(v_0)|}\frac{|S_{i_m}(v)\cap
S_l(v_0)|}{\prod_{j =1}^{m}| S_{i_j} ( v_{j -1} ) |}x_l\\
&=\sum_{l\in
I(\Gamma)}\sum_{k\in
I(\Gamma)}\sum_{v_{m-1}\in
S_{i_{m-1}}(v_{m-2})\cap
S_k(v_0)}\sum_{v\in
S_k(v_0)}\frac{1}{|S_k(v_0)|}\frac{|S_{i_m}(v_{m-1})\cap
S_l(v_0)|}{\prod_{j =1}^{m}| S_{i_j} ( v_{j -1} ) |}x_l\\
&=\sum_{l\in
I(\Gamma)}\sum_{k\in
I(\Gamma)}\sum_{v_{m-1}\in
S_{i_{m-1}}(v_{m-2})\cap
S_k(v_0)}\frac{|S_{i_m}(v_{m-1})\cap
S_l(v_0)|}{\prod_{j =1}^{m}| S_{i_j} ( v_{j -1} ) |}x_l\\
&=\sum_{l\in
I(\Gamma)}\sum_{v_{m-1}\in
S_{i_{m-1}}(v_{m-2})}\frac{|S_{i_m}(v_{m-1})\cap
S_l(v_0)|}{\prod_{j =1}^{m}| S_{i_j} ( v_{j -1} ) |}x_l\\
&=\sum_{l\in
I(\Gamma)}\sum_{v_{m-1}\in
S_{i_{m-1}}(v_{m-2})}\sum_{v_m\in
S_{i_m}(v_{m-1})\cap
S_l(v_0)}\frac{1}{\prod_{j =1}^{m}| S_{i_j} ( v_{j -1} ) |}x_l\\
&=\sum_{v_{m-1}\in
S_{i_{m-1}}(v_{m-2})}\sum_{l\in
I(\Gamma)}\sum_{v_m\in
S_{i_m}(v_{m-1})\cap
S_l(v_0)}\frac{1}{\prod_{j =1}^{m}| S_{i_j} ( v_{j -1} ) |}x_{|v_m|}\\
&=\sum_{v_{m-1}\in
S_{i_{m-1}}(v_{m-2})}\sum_{v_m\in
S_{i_m}(v_{m-1})}\frac{1}{\prod_{j =1}^{m}| S_{i_j} ( v_{j -1} ) |}x_{|v_m|},\\
\end{align*}
and hence 
\begin{align*}
&\sum_{v_1 \in S_{i_1}(v_0)}\sum_{v_2 \in S_{i_2}(v_1)} \cdots \sum_{v_{m-2} \in S_{i_{m-2}}(v_{m-3})}\sum_{k\in
I(\Gamma)}\frac{|S_{i_{m-1}}(v_{m -2})\cap
S_k(v_0)|}{\prod_{j =1}^{m-1}| S_{i_j} ( v_{j -1} ) |} x_{k}\circ
x_{i_m}\\
&=\sum_{v_1 \in S_{i_1}(v_0)}\sum_{v_2 \in S_{i_2}(v_1)} \cdots \sum_{v_{m-2} \in S_{i_{m-2}}(v_{m-3})}\sum_{v_{m-1}\in
S_{i_{m-1}}(v_{m-2})}\sum_{v_m\in
S_{i_m}(v_{m-1})}\frac{1}{\prod_{j =1}^{m}| S_{i_j} ( v_{j -1} ) |}x_{|v_m|}.
\end{align*}
This completes the proof.
\end{proof}
%%%%%%%%%%%%%%%%%%%%%%%%%%%%%%%%%%%%%%%%%%%%%%
%%%%%%%%%%%%%%%%%%%%%%%%%%%%%%%%%%%%%%%%%%%%%%%%%%%%%%%%%%%%%%%%%%%%%%%%%%%%%
\begin{exam}
All distance regular graphs satisfy the assumption {\rm(}the conditions $(S1)$ and $(S2)${\rm)} in the previous theorem. We already know that the pointed graph $(\Gamma,w_0)$ in {\rm
Example \ref{someg}} is not distance regular and satisfies the conditions $(S1)$ and $(S2)$. We present other examples of such pointed graphs drown as {\rm
Figure $5$} and {\rm
Figure $6$}. 
They all produce hypergroups.
\[
\begin{xy}
(0,0)*{\circ}="A",
(-9,-5)*{\circ}="B",
(9,-5)*{\circ}="C",
(-13,-12)*{\circ}="D",
(13,-12)*{\circ}="E",
(-12,-19)*{\circ}="F",
(12,-19)*{\circ}="G",
(-7,-25)*{\circ}="H",
(7,-25)*{\circ}="I",
(40,0)*{\circ}="K",
(32,-15)*{\circ}="L",
(48,-15)*{\circ}="M",
(29,-11)*{\circ}="N",
(51,-11)*{\circ}="O",
(40,-8)*{\circ}="P",
(27,-25)*{\circ}="Q",
(53,-25)*{\circ}="R",
{ "A" \ar @{-} "B" },
{ "B" \ar @{-} "D" },
{ "A" \ar @{-} "C" },
{ "C" \ar @{-} "E" },
{ "G" \ar @{-} "I" },
{ "D" \ar @{-} "F" },
{ "G" \ar @{-} "E" },
{ "H" \ar @{-} "F" },
{ "H" \ar @{-} "I" },
{ "A" \ar @{-} "F" },
{ "A" \ar @{-} "G" },
{ "F" \ar @{-} "I" },
{ "G" \ar @{-} "H" },
{ "B" \ar @{-} "E" },
{ "B" \ar @{-} "I" },
{ "C" \ar @{-} "D" },
{ "C" \ar @{-} "H" },
{ "E" \ar @{-} "D" },
{ "K" \ar @{-} "L" },
{ "K" \ar @{-} "M" },
{ "M" \ar @{-} "L" },
{ "L" \ar @{-} "N" },
{ "M" \ar @{-} "O" },
{ "K" \ar @{-} "P" },
{ "O" \ar @{-} "P" },
{ "P" \ar @{-} "N" },
{ "K" \ar @{-} "O" },
{ "K" \ar @{-} "N" },
{ "Q" \ar @{-} "L" },
{ "Q" \ar @{-} "M" },
{ "Q" \ar @{-} "N" },
{ "Q" \ar @{-} "O" },
{ "Q" \ar @{-} "P" },
{ "R" \ar @{-} "L" },
{ "R" \ar @{-} "M" },
{ "R" \ar @{-} "N" },
{ "R" \ar @{-} "O" },
{ "R" \ar @{-} "P" },
(0,4)*{v_0'},
(40,4)*{v_0''},
(0,-30)*{\mbox{{\rm
Figure }}5},
(40,-30)*{\mbox{{\rm
Figure }}6}
\end{xy}
\]
On the other hand, 
the authors could not find any example of pointed graphs not producing hypergroups and satisfying the conditions $(S1)$ and $(S2)$.
We propose the following conjecture:
\begin{conj}
The conditions $(S1)$ and $(S2)$ imply producing hypergroups.
\end{conj}
\end{exam}
%%%%%%%%%%%%%%%%%%%%%%%%%%%%%%%%%%%%%%%%%%%%%%%%%%%%%%%%%%%%%%%%%%%%%%%%%%%%%%%
\begin{exam}
The $3$-gonal prism graph $\mathcal{P}_3$ in {\rm
Example \ref{prism}} satisfies
\[
PL(1,2,1)=\frac{6}{27}x_0+\frac{10}{27}x_1+\frac{11}{27}x_2\neq\frac{2}{9}x_0+\frac{1}{3}x_1+\frac{4}{9}x_2=J(1,2,1).
\]
The binary tree $\mathcal{B}$ with the base point $v_0$ of $\mathcal{B}$ defined in {\rm
Example \ref{tree}}, satisfies
\[
PL(1,1,2)=\frac{1}{9}x_0+\frac{4}{9}x_2+\frac{4}{9}x_4\neq\frac{1}{6}x_0+\frac{1}{6}x_2+\frac{2}{3}x_4=J(1,1,2).
\] 
\end{exam}

%%%%%%%%%%%%%%%%%%%%%%%%%%%%%%%%%%%%%%%%%%%%%%%%%%%%%%%%%%%%%%%%%%%%%%%%%%%%%%

%%%%%%%%%%%%%%%%%%%%%%%%%%%%%%%%%%%%%%%%%%%%%%%%%%%%%%%%%%%%%%%%%%%%%%%%%%%% 
\begin{coro}
If $G$ is a Cayley graph equipped with the condition {\rm
(S2)} then we have $PL(i_1,\ldots,i_m)= J(
i_1,\ldots,i_m)=\sum_{k\in
I(\Gamma)}p_{i_1,\ldots,i_m}^kx_k$ for all $k,i_1,\ldots,i_m\in
I(\Gamma)$. We have also $p_{i_{\sigma(1)},\ldots,i_{\sigma(m)}}^k=p_{i_1,\ldots,i_m}^k$ for every permutation $\sigma\in\mathfrak{S}_m$ and all $k,i_1,\ldots,i_m\in
I(G)$.
\end{coro}
\begin{proof}
By Lemma \ref{cayleylemm2}, $G$ is distance regular and produces a hypergroup $H(G)$. Thus, Theorem \ref{2-3} and Theorem \ref{1-2} imply the first assertion, and the commutativity and the associativity of $H(G)$ imply the second one.
\end{proof}
%%%%%%%%%%%%%%%%%%%%%%%%%%%%%%%%%%%%%%%%%%%%%%%%%%%%%%%%%%%%%%%%%%%%%%%%%%%%%%
\begin{exam}
For the Cayley graph ${\rm
Cay}(\mathbb{Z}^2)$, we have $PL(1,2,3)\neq
PL(2,3,1)$.
\end{exam}
%%%%%%%%%%%%%%%%%%%%%%%%%%%%%%%%%%%%%%%%%%%%%%%%%%%%%%%%%%%%%%%%%%%%%%%%%%%%%%

In the end of this section, we shall present a formula with respect to the transition probabilities giving the Markov chain discussed in Section $3$, and constants of the pre-hypergroup structure derived from $G$ as follows:
%%%%%%%%%%%%%%%%%%%%%%%%%%%%%%%%%%%%%%%%%%%%%%%%%%%%%%%%%%%%%%%%%%%%%%%%%%%%%%%%%
\begin{prop}\label{markovpijk}
Let $G$ be a Cayley graph which is not necessary finite and $\{p_{i,j}^k\}_{i,j,k\in
I(G)}$ the constants giving the structure of the pre-hypergroup $H(G)$ derived from $G$. For $i,j\in
I(G)$ and $n\in\mathbb{N}$, we have $\mathbb{P}\left(Z_{2}=j ~\bigg{|}~
Z_1=i\right)=\sum_{k\in
I(G)}p_{i,k}^j\alpha_k|S_k|.$ When $G$ is finite and has the uniform distribution, we have $|S_j|=\sum_{k=0}^{M(G)}p_{i,k}^j|S_k|$. 
\end{prop}
\begin{proof}
The proof is straightforward.
\end{proof}

\section{Transition matrices associated with hypergroups derived from graphs}\label{Realization of hypergroups as operator algebras}
In this section, we will identify the hypergroup derived from a pointed graph with a commutative algebra  whose basis is transition matrices given by products in the hypergroup, and estimate the operator norms of the transition matrices. We will clarify a relationship between random walks and products of the transition matrices.

A finite hypergroup induces a finite dimensional unital algebla which can be identified with a matrix algebra.
This fact is true for the discrete infinite case.
In other words, for a not necessary finite discrete hypergroup $H=\{x_i\}_{i \in I(H)}$ with a structure constant $\{p_{i,j}^k\}_{i,j,k\in I(H)}$, 
a family $\mathcal{P}_H=\{P_k\}_{k\in I(H)}$ of transition matrices $P_k = (p_{k,i}^j)_{i,j \in I (H)}$ satisfies that $P_iP_j=\sum_{k\in
I(H)}p_{i,j}^kP_k$ for all $i,j \in I(H)$.

\begin{rema}
The associativity of the hypergroup $H(\Gamma,v_0)$ derived from a pointed graph $(\Gamma,v_0)$ is characterized by the commutativity of transition matrices $\mathcal{P}_{H(\Gamma,v_0)}$. That is, the commutative pre-hypergroup $H(\Gamma,v_0)$ forms a hypergroup if and only if all of transition probability matrices in $\mathcal{P}_{H(\Gamma,v_0)}=\{P_k\}_{k\in
I(\Gamma,v_0)}$ mutually commute.
\end{rema}
%%%%%%%%%%%%%%%%%%%%%%%%%%%%%%%%%%%%%%%%%%%%%%%%%%%%%%%%%%%%%%%%%%%%%%%%%%
We can also define the transition matrices $\mathcal{P}_H$ from a pre-hypergroup $H$ by the same way. Then, matrices in $\mathcal{P}_H$ can be regarded as linear operators on the Hilbert space $\ell^2(H):=\{(\xi_n)_{n\in
I(H)}\mid\xi_n\in\mathbb{C},\ \sum_{n\in
I(H)}|\xi_n|^2<\infty\}$ as follows: for $k\in
I(H)$, we define an operator, which is denote by $P_k$ too, on $\ell^2(H)$ by $P_k(\xi)_n
=\sum_{l\in
I(H)}p_{k,l}^n\xi_l$ for $\xi=(\xi_n)\in\ell^2(H)$ with $\sum_{n\in
I(H)}|\sum_{l\in
I(\Gamma)}p_{k,l}^n\xi_l|^2<\infty$.
The actions can be regarded as the matrix products of row vectors in $\ell^2(H)$ and matrix $P_k$'s.

If $\Gamma=G$ is a finite Cayley graph with the uniform distribution, then Proposition \ref{markovpijk} implies that the distribution $\pi_G$ defined as \eqref{s.d.} is a stationary distribution of all transition matrices in $\mathcal{P}_{H(G)}$. 
However, in the infinite case, $P_k$ does not always have a stationary distribution. Indeed, the transition matrix $P_1$ associated with the hypergroup $H(\mathbb{Z})$ in Example \ref{freegroup} has no stationary distribution. For the irreducibility, 
there are the case in which a transition matrix $P_k$ associated with a hypergroup derived from a graph is irreducible and the case in which it is reducible as follows.
%%%%%%%%%%%%%%%%%%%%%%%%%%%%%%%%%%%%%%%%%%%%%%%%%%%%%%%%%%%%%%%%%%%%%%%%%%%%%%%%%
\begin{exam}
Let $P_1,P_2$ be the transition matrices associated with the hypergroup $H(\mathcal{C}_4)$ derived from the $4$-cycle graph $\mathcal{C}_4={\rm
Cay}(\mathbb{Z}/4\mathbb{Z},\{\pm1\})$.
Then $P_1$ is irreducible and $P_2$ is reducible.
\end{exam}
%%%%%%%%%%%%%%%%%%%%%%%%%%%%%%%%%%%%%%%%%%%%%%%%%%%%%%%%%%%%%%%%%%%%%%%%%%%

Now, we shall estimate the operator norm of $P_k$ associated with the pre-hypergroup derived from a pointed graph. We define sets
\begin{align*}
&{\rm
Supp}(k)=\{(i,j)\in
I(\Gamma)^2\mid
p_{k,i}^j \neq0\}\\
&{\rm
Supp}_i(k)=\{j\in
I(\Gamma)\mid
p_{k,i}^j \neq0\}\subset\{|i-k|,|i-k|+1,\ldots,i+k\},\\
&{\rm
Supp}^j(k)=\{i\in
I(\Gamma)\mid
p_{k,i}^j \neq0\}\subset\{|j-k|,|j-k|+1,\ldots,j+k\}.
\end{align*}
Then, we have the following theorem.
%%%%%%%%%%%%%%%%%%%%%%%%%%%%%%%%%%%%%%%%%%%%%
\begin{theo}\label{norm}
Let $\Gamma$ be a pointed graph. For all $k\in
I(\Gamma)$, the operator $P_k$ is a bounded operator on $\ell^2(H(\Gamma))$. Moreover, if we define constants $c_k=\sup_{j\in
I(\Gamma)}\sum_{i\in{\rm
Supp}^j(k)}(p_{k,i}^j)^2$ and $d_k=\sup_{i\in
I(\Gamma)}|{\rm
Supp}_i(k)|$ the operator norm of $P_k$ is estimated as $1\leq\|P_k\|\leq
\sqrt{c_kd_k}$.
\end{theo}

\begin{proof}
For every $\xi=(\xi_n)\in\ell^2(H(\Gamma))$, we have
\begin{align}
&\|P_k(\xi)\|^2=\sum_{j\in
I(\Gamma)}\left|\sum_{i\in
I(\Gamma)}p_{k,i}^j\xi_i\right|^2=\sum_{j\in
I(\Gamma)}\left|\sum_{i\in{\rm
Supp}^j(k)}p_{k,i}^j\xi_i\right|^2\label{inne}\\
&\leq\sum_{j\in
I(\Gamma)}\left(\sum_{i\in{\rm
Supp}^j(k)}(p_{k,i}^j)^2\right)\left(\sum_{i\in{\rm
Supp}^j(k)}|\xi_i|^2\right)\leq
c_k\sum_{j\in
I(\Gamma)}\sum_{i\in{\rm
Supp}^j(k)}|\xi_i|^2.\nonumber
\end{align}
Interchanging of the order of the above sum, we have
\begin{align*}
&\sum_{j\in
I(\Gamma)}\sum_{i\in{\rm
Supp}^j(k)}|\xi_i|^2=\sum_{(i,j)\in{\rm
Supp}(k)}|\xi_i|^2=\sum_{i\in
I(\Gamma)}\sum_{j\in{\rm
Supp}_i(k)}|\xi_i|^2\leq
d_k\sum_{i\in
I(\Gamma)}|\xi_i|^2.
\end{align*}
By \eqref{inne}, the inequality $\|P_k\|\leq
\sqrt{c_kd_k}$ has been proved. Also, if $\xi=(1,0,\ldots)\in\ell^2(H(\Gamma))$, we have $\|P_k(\xi)\|=1$, and hence $\|P_k\|\geq1$ for all $k\in
I(H)$.
\end{proof}

\begin{exam}
The transition matrix $P_1$ associated with the hypergroup $H(\mathbb{Z})$ has the norm greater than $1$.
Indeed, taking $\xi_n = \frac{1}{2^n}$ for each $n \in \mathbb{N}_0$,
the vector $\xi = (\xi_n) \in \ell^2(H(\mathbb{Z})) $ has the norm $\frac{2}{\sqrt{3}}$ and $ \| P_1(\xi) \| = \sqrt{2}$, 
and hence we have $  \| P_1 \| \geq \sqrt{ \frac{3}{2} }>1 $.
\end{exam}
%%%%%%%%%%%%%%%%%%%%%%%%%%%%%%%%%%%%%%%%%%%%%%%%%%%%%%%%%%%%%%%%%%%%%%%%%%%%%%%%
We shall discuss the uniformly boundedness for the operators $P_k\ (k\in
I(\Gamma))$ for a pointed infinite graph $\Gamma$ as the following corollary. (Obviously, when $\Gamma$ is a finite graph the set $\{\|P_k\|\}_{k\in
I(\Gamma)}$ is bounded.)

%%%%%%%%%%%%%%%%%%%%%%%%%%%%%%%%%%%%%%%%%%%%%%%%%%%%%%%%%%%%%%%%%%%%%%%%%%%%%%%%
\begin{coro}\label{uniformb}
Let $\Gamma$ be an infinite graph with a vertex set $V$ and $v_0\in
V$ a base point. If $S(\Gamma):=\sup_{v\in
V}\sup_{k\in
I(\Gamma)}|S_k(v)|<\infty$ then we have $\|P_k\|\leq
S(\Gamma)^2$.
\end{coro}
\begin{proof}
It is enough to show that $c_k,d_k\leq
S(\Gamma)^2$ by Theorem \ref{norm}. For $i,j,k\in
I(\Gamma)$, we have
\begin{align*}
&{\rm
Supp}^j(k)=\hspace{-2pt}\bigcup_{v\in
S_k(v_0)}\{l\in\mathbb{N}_0\mid
S_l(v)\cup
S_j(v_0)\neq\varnothing\}=\hspace{-2pt}\bigcup_{v\in
S_k(v_0)}\bigcup_{w\in
S_j(v_0)}\{l\in\mathbb{N}_0\mid
w\in
S_{l}(v)\},\\
&{\rm
Supp}_i(k)=\hspace{-2pt}\bigcup_{v\in
S_k(v_0)}\{m\in\mathbb{N}_0\mid
S_i(v)\cup
S_m(v_0)\neq\varnothing\}=\hspace{-2pt}\bigcup_{v\in
S_k(v_0)}\bigcup_{w\in
S_i(v)}\{m\in\mathbb{N}_0\mid
w\in
S_m(v_0)\}
\end{align*}
by the definition of $p_{k,i}^j$. These imply that $|{\rm
Supp}^j(k)|\leq\sum_{v\in
S_k(v_0)}\sum_{w\in
S_j(v_0)}1=S(\Gamma)^2$ and $|{\rm
Supp}_i(k)|\leq\sum_{v\in
S_k(v_0)}|S_i(v)|\leq
S(\Gamma)^2$, and hence we have $c_k,d_k\leq
S(\Gamma)^2$.
\end{proof}
%%%%%%%%%%%%%%%%%%%%%%%%%%%%%%%%%%%%%%%%%%%%%%%%%%%%%%%%%%%%%%%%%%%%%%%%%%%%%%%

\begin{exam}
It is easy to check that the $1$-dimensional lattice ${\rm
Cay}(\mathbb{Z},\{\pm1\})$ and the infinite ladder graph $\mathcal{L}={\rm
Cay}(\mathbb{Z}\oplus(\mathbb{Z}/2\mathbb{Z}),\{(\pm1,0),(0,1)\})$ satisfy the assumption of Corollary \ref{uniformb} and $S(\mathbb{Z})=2,\ S(\mathcal{L})=4$. 
\end{exam}
%%%%%%%%%%%%%%%%%%%%%%%%%%%%%%%%%%%%%%%%%%%%%%%%%%%%%%%%%%%%%%%%%%%%%%%%%%%%%%%
%%%%%%%%%%%%%%%%%%%%%%%%%%%%%%%%%%%%%%%%%%%%%%%%%%%%%%%%%%%%%%%%%%%%%%%%%%%
\begin{coro}\label{maincoro}
Let $(\Gamma,v_0)$ be a pointed graph producing a hypergroup $H(\Gamma)$ and $\mathcal{P}_{H(\Gamma)}$ the transition matrices associated with the hypergroup $H(\Gamma)$. If $\Gamma$ satisfies the conditions $(S1)$ and $(S2)$ then we have $(P_{i_1}P_{i_2}\cdots
P_{i_m})_{i,j}=\sum_{k\in
I(G)}\tilde{p}_{i_1,i_2,\ldots,i_m}^kp_{k,i}^j$ for all $m\in\mathbb{N}$ and $i,j,i_1,\ldots,i_m\in
I(\Gamma)$.
\end{coro}
%%%%%%%%%%%%%%%%%%%%%%%%%%%%%%%%%%%%%%%%%%%%%%%%%%%%%%%%%%%%%%%%%%%%%%%%%%%%

Under the assumption of the previous corollary, the $k$-th coefficient in $P_{i_m}\cdots
P_{i_1}\xi^0$ coincides with $\tilde{p}_{i_1,\ldots,i_m}^k$, where $\xi^0=(1,0,0,\ldots)$. In other words, a matrix product of $P_k$'s describes a distribution of distances between and the base point and a vertex to which a random walker reaches from the base point by some steps (see Remark \ref{probrema} (2)).

\subsection*{Acknowledgements}
The authors would like to express gratitude to Satoshi Kawakami for his helpful comments. 
They also would like to thank Tomohiro Ikkai, Yuta Suzuki, Fujie Futaba for their useful discussion and interest.
This work was supported by JSPS KAKENHI Grant-in-Aid for Research Activity Start-up (No. 19K23403).


\begin{thebibliography}{99}
\bibitem{bloo-heye95}
W. R. Bloom, H. Heyer, {\it
Harmonic analysis of probability measures on hypergroups},
Walter de Gruyter \& Co., 1995.
\bibitem{bose-mesn59}
R. C. Bose, D. M. Mesner, On linear associative algebras corresponding to association schemes of partially balanced design, {\it
Ann. Math. Statist.} 30 (1959), 21-38.
\bibitem{brem99}
P. Br\'{e}maud, {\it
Markov chains: Gibbs fields, Monte Carlo simulation, and queues}, Springer, New York, 1999.
\bibitem{dkt16}
E. R. van Dam, J. H. Koolen, H. Tanaka, 
Distance-regular graphs,
{\it Electron. J. Combin.}, $\#$DS22 (2016), 
\url{https://www.combinatorics.org/ojs/index.php/eljc/article/view/DS22}.
\bibitem{bcn89}
A. E. Brouwer, A. M. Cohen, A. Neumaier, {\it
Distance-regular graphs}, Springer-Verlag, 1989. 
\bibitem{dunk73}
C. F. Dunkl, The measure algebra of a locally compact hypergroup, {\it
Trans. Amer. Math. Soc.} 179 (1973), 331-348.
\bibitem{gods80}
C. Godsil, More odd graph theory, {\it
Discrete Math.} 32 (1980), 205-207.
\bibitem{gods-royl01}
C. Godsil, G. Royle, {\it
Algebraic graph theory}, Springer, 2001
\bibitem{jewe75}
R. I. Jewett, Spaces with an abstract convolution of measures, {\it
Adv. in Math.} 18 (1975), 1-101.
\bibitem{kreb-shah11}
M. Krebs, A. Shaheen, {\it
Expander Families and Cayley graphs}, Oxford University Press, 2011.
\bibitem{ikka-sawa19}
T. Ikkai, Y. Sawada, Hypergroups derived from random walks on some infinite graphs, {\it
Monatsh. Math.} 189 (2019), 321-353.
\bibitem{lass83}
R. Lasser, Orthogonal polynomials and hypergroups, {\it
Rend. Mat. Appl.} 3 (1983), 185-209.
\bibitem{lass05}
R. Lasser, Discrete commutative hypergroups, {\it Proceedings of the SIAM Sumer School}, 2005.
\bibitem{spec78}
R. Spector, Mesures invariantes sur les hypergroupes, {\it
Trans. Amer. Math. Soc.} 239 (1978), 147-165.
\bibitem{stmoy17}
A. Suzuki, T. Tsurii, Y. Matsuzawa, H. Ohno, S. Yamanaka, Noncommutative hypergroup of order five, {\it
J. Algebra Appl.} 16 (2017), 21pp.
\bibitem{wild94}
N. J. Wildberger, Hypergroups associated to random walks on Platonic solids, Preprint, Univ.\ of NSW, 1994. 
\bibitem{wild95}
N. J. Wildberger, Finite commutative hypergroups and applications from
group theory to conformal field theory, {\it
Contemp. Math.} 183 (1995), 413-434.
\bibitem{wild02}
N. J. Wildberger, Strong hypergroup of order three, {\it
J. Pure Appl. Algebra} 174 (2002), 95-115.
\end{thebibliography}
\end{document}